\documentclass[12pt,twoside,final,psamsfonts]{amsart}
\usepackage{times,amssymb,hyperref}

\theoremstyle{plain}
\newtheorem*{theorem*}{Theorem}
\newtheorem*{theoremA}{Theorem A}
\newtheorem*{theoremB}{Theorem B}
\newtheorem{theorem}{Theorem}
\newtheorem{lemma}[theorem]{Lemma}
\theoremstyle{definition}
\newtheorem*{definition}{Definition}

\newcommand{\C}{\mathbb{C}}
\renewcommand{\H}{\mathbb{H}}
\newcommand{\N}{\mathbb{N}}
\newcommand{\R}{\mathbb{R}}
\newcommand{\E}{\mathcal{E}}
\newcommand{\B}{\mathcal{B}}
\newcommand{\EP}{\hat{\E^\prime}(\R)}

\newcommand{\Int}{\operatorname{Int}}
\newcommand{\Hol}{\operatorname{Hol}}
\newcommand{\lil}{\lambda\in\Lambda}
\renewcommand{\Im}{\operatorname{Im}}

\title[Interpolation in $\mathcal B$]
{Interpolation sequences for the Bernstein algebra}

\author[X. Massaneda]{Xavier Massaneda}
\address{Departament de Matem\`atica Aplicada i An\`alisi,
Universitat  de Bar\-ce\-lo\-na, Gran Via 585, 08071-Bar\-ce\-lo\-na, Spain}
\email{xavier.massaneda@ub.edu}

\author[J. Ortega-Cerd\`a]{Joaquim Ortega-Cerd\`a}
\address{Departament de Matem\`atica Aplicada i An\`alisi,
Universitat  de Bar\-ce\-lo\-na, Gran Via 585, 08071-Bar\-ce\-lo\-na, Spain}
\email{jortega@ub.edu}

\thanks{Both authors supported by DGICYT grant 
MTM2005-008984-C02-02 and the CIRIT grant 2005-SGR 00611.}

\date{\today}



\begin{document}

\maketitle
\begin{abstract} 
We give a description, in analytic and geometric terms, of the interpolation sequences for the algebra of entire functions  of exponential type which are bounded on the real line.
\end{abstract}

\begin{center}
\small\textit{Dedicated to Victor Petrovich Havin in his 75 birthday}
\end{center}

\section{Introduction}

In this paper we describe interpolation sequences for the \emph{Bernstein algebra} $\B$  of  entire functions  of exponential type which are bounded on the real line. 

The space $\B$ consists of the entire functions $f$ such that for some $A,B>0$
\[
\log|f(z)|\leq A+B|\Im z|\qquad z\in\C\ .
\]
It can also be viewed as the union of the \emph{Bernstein spaces} $B_\sigma$, $\sigma>0$, of entire functions $f$ such that
\begin{equation*}
\sup_{z\in\C} |f(z)| e^{-\sigma|\Im z|}<\infty\ ,
\end{equation*}
which are precisely the bounded functions with Fourier-Laplace transforms supported in $[-\sigma,\sigma]$.

We were motivated to study the interpolation in $\B$ by the recent characterisation of zero sets in the same class obtained by Sergei Ju. Favorov. 

Given a discrete sequence $\Lambda\subset\C$ and a point $z\in\C$, consider the counting function  $n_\Lambda(z,t)=\#(\Lambda\cap \overline{D(z,t)})$, $t>0$.

\begin{theorem*}\cite[Theorem 2.2]{Fa} A sequence
$\Lambda$ is the zero set of a function in $\B$ if and only if:
\begin{enumerate}
\item [(a)] $\exists \lim\limits_{R\to\infty}\sum\limits_{\lambda\in \Lambda\cap D(0,R)\setminus\{0\}} \lambda^{-1}$
\item [(b)] $n_\Lambda(0,t)=O(t)$, $t\to\infty$
\item [(c)] $n_\Lambda(0,t+1)-n_\Lambda(0,t)=o(t)$, $t\to\infty$
\item [(d)] $\sup\limits_{x\in\R}
\displaystyle{
\int_{0}^\infty [n_\Lambda(b,t)-n_\Lambda(x,t)]\frac{dt}t}<\infty$ for some $b\in\R\setminus\Lambda$.
\end{enumerate}
\end{theorem*} 

Notice that a subset of a zero set $\Lambda$ of $\B$ is not necessarily a zero set of $\B$, since  conditions (a) and (d) imply a certain regularity in the distribution of $\Lambda$ that can be destroyed when removing points.

In view of this resulted we ask ourselves whether interpolation sequences for $\B$ can be described in similar terms.
Given the characteristic growth of functions in $\B$, the natural interpolation problem for this algebra is the following.

\begin{definition} A discrete sequence $\Lambda$ in $\C$ is called an \emph{interpolation sequence} for
$\B$ if for every sequence of values $\{v_\lambda\}_{\lil}$ with  
\begin{equation}\label{values}
\sup_{\lil} |v_\lambda|
e^{-C|\Im\lambda|} <\infty 
\end{equation}
for some $C>0$, there exists $f\in \B$ with 
\[
f(\lambda)=v_\lambda\qquad \lil\ \ .
\]
\end{definition}

Letting $\B(\Lambda)$ denote the space of sequences 
$\{v_\lambda\}_{\lil}$ satisfying \eqref{values} for some $C>0$, we
can equivalently define interpolating sequences as those such that the
restriction operator
\[ 
\begin{split}
\mathcal R_\Lambda : \B &\longrightarrow \B(\Lambda)\\
 f &\ \mapsto\; \{f(\lambda)\}_{\lil}
\end{split} 
\]
is onto.

Applying the open mapping theorem to $\mathcal R_\Lambda$ \cite[Lemma 2.2.6]{Br-Ga} one sees that the interpolation is stable in the following sense: for each $C>0$ there exist $M,\sigma>0$ such that if $\{v_\lambda\}_{\lil}$ satisfies \eqref{values} then there exists $f\in\B$ with $f(\lambda)=v_\lambda$, $\lil$, and $\sup_{z\in\C}|f(z)|e^{-\sigma|\Im z|}\leq M \sup_{\lil} |v_\lambda| e^{-C|\Im\lambda|}$.

There is an alternative (and equivalent) definition of $\B$-interpolation:

\begin{definition}  A sequence $\Lambda$ is \emph{free interpolation} (denoted
$\Lambda \in Int_{\ell^\infty}\,\B$) if the trace $\mathcal R_\Lambda(\B)$ is \emph{ideal}, that is, if $\ell^\infty\cdot \mathcal R_\Lambda(\B)\subset \mathcal R_\Lambda(\B)$. 

Thus,  $\Lambda \in Int_{\ell^\infty}\,\B$ if and only if whenever $f\in\B$ and $\{\alpha_\lambda\}\in\ell^\infty$ there exists $g\in\B$ such that $g(\lambda)=\alpha_\lambda f(\lambda)$, $\lil$. Since $\B$ is an algebra, this happens if and only if
for any sequence of values $\{\alpha_\lambda\}_{\lil}\in\ell^\infty$ there exists
$f\in \B$ with  $f(\lambda)=\alpha_\lambda$, $\lil$.
\end{definition}

There exists a general theory (and in particular an analytic description of interpolating sequences) for the algebras
\[
 A_p=\bigl\{f\in\Hol(\C) :\ \exists C>0\, :\, \sup_{z\in\C}|f(z)|e^{-Cp(z)}<\infty\bigr\}\ ,
\]
where $p:\C\longrightarrow\R_+$ is a subharmonic weight such that:
\begin{itemize}
 \item [(a)] $\log(1+|z|^2)=O(p(z))$,
 \item[(b)] there are constants $C,D>0$ such that whenever $|\zeta-z|\leq 1$ then $p(\zeta)\leq Cp(z)+D$.
\end{itemize}

For this we refer to \cite[Chapter 2]{Br-Ga}.

Notice that $\B$ does not fall into this general theory, since $p(z)=|\Im z|$ does not satisfy (a) (in particular, $\B$ does not contain polynomials). However, the techniques used in the description of $A_p$-interpolation mentioned above (see \cite[Corollary 3.5]{Br-Li}) can be adapted to obtain an analogous result for $\B$ .

\begin{theoremA}\label{analytic} 
The following statements are equivalent:

\begin{enumerate}
\item[(a)] $\Lambda \in Int\,\B$

\item[(b)] $\Lambda \in Int_{\ell^\infty}\,\B$

\item[(c)] There exist $A,B,C>0$ and peak-functions $f_\lambda\in \B$ such that $f_\lambda(\lambda')=\delta_{\lambda,\lambda'}$ and
\[
\sup_{z\in \C} |f_{\lambda}(z)|e^{-C |\Im z|}\leq A e^{B|\Im \lambda|}\ .
\]

\item[(d)] There exist $F\in \B$ and $\varepsilon,C>0$ such that $F(\lambda)=0$ for all $\lil$ and
\[
|F'(\lambda)|\geq \varepsilon e^{-C|\Im \lambda|}\ .
\]

\end{enumerate}

\end{theoremA}

As soon as the analytic conditions of Theorem A hold we can proceed as in \cite{MOO} and
obtain  geometric conditions. That paper deals with $A_p$-interpolation for weights of the form
\[
 p(z)=|\Im z|+\omega(|z|)\ ,
\]
being $\omega(t)$  a subadditive increasing continuous function, normalized with $\omega(0)=0$ such that $\log(1+t)\lesssim \omega(t)$ and $\int_0^\infty\frac{\omega(t)}{1+t^2}dt<\infty$. Taking $p(z)=|\Im z|$ (i.e. $\omega\equiv0$) and using Theorem A we can mimic the proof of \cite[Theorem 1]{MOO} and obtain a geometric characterisation.

In order to state this result consider the counting function $n_\Lambda(z,r)$ defined previously
and the integrated version
\[
N_\Lambda(z,r)=\int_0^r\frac{n_\Lambda(z,t)-n_\Lambda(z,0)}t\, dt + n_\Lambda(z,0)\log r .
\]

Hereinafter we assume that $\Lambda\cap \R=\emptyset$; otherwise replace $\R$ by any the horizontal line 
that does not contain any of the points in $\Lambda$.

\begin{theoremB}\label{teorema}
The sequence $\Lambda$ is $\B$-interpolation if 
and only if:
\begin{itemize}
\item[(i)] There is $C>0$ such that 
\begin{equation*}
N_\Lambda(\lambda,  |\Im \lambda| ) \le C |\Im \lambda|\qquad \forall 
\lambda\in\Lambda.
\end{equation*}

\item[(ii)] The following Carleson-type condition holds
\[
\sup_{x\in\R}
\sum\limits_{\lambda\in\Lambda}
\frac{|\Im\, \lambda|}{|x-\lambda|^2}
<\infty .
\]
\end{itemize}
\end{theoremB}

Since the Poisson kernel at $\lambda$ in the corresponding half-plane  (upper
half-plane if $\Im\,\lambda>0$ and lower half-plane when  $\Im\,\lambda<0$) is
$P(\lambda,x)=\frac{|\Im\, \lambda|}{|x-\lambda|^2}$,  a restatement of
condition (ii) is that the measure  $\sum\limits_{\lambda\in\Lambda}
\delta_\lambda$ has bounded Poisson balayage.

The paper is structured as follows. In Section~\ref{analytic-conditions} we prove Theorem A and describe other analytic properties of $\B$-interpolation sequences which are relevant in the proof of Theorem B. For the sake of completeness, we briefly recall the proof of Theorem B in 
Section~\ref{sufsection}. 

A final remark about notation.  $C$ will always denote a positive constant and 
its actual value may change from one  occurrence to the next. $A= O(B)$ and 
$A\lesssim B$ mean that $A\leq c B$ for some $c>0$, and $A\simeq B$ is
$A\lesssim B\lesssim A$.

\section{Proof of the analytic conditions}\label{analytic-conditions}

Before the proof of Theorem A we need to spell out some of its consequences.

\begin{lemma}\label{weak-separation} Let $\Lambda$ be a discrete sequence in $\C$ for which (c) in Theorem A holds.
Then
\begin{itemize}
\item[(i)] There exist $\alpha,\varepsilon>0$ such that the discs $D_\lambda:=D(\lambda, \delta_\lambda)$, with
$\delta_\lambda=\varepsilon e^{-\alpha|\Im \lambda|}$, are pairwise disjoint.

\item[(ii)] There exist $C>0$ such that for all $\beta>2$
\[
\sup_{z\in\C}
\sum_{\lil}\frac{ e^{-C|\Im \lambda|}}{1+|z-\lambda|^\beta}<\infty\ .
\]
\end{itemize}
\end{lemma}

When $\Lambda$ satisfies (i) we say that $\Lambda$ is \emph{weakly separated}. 

\begin{proof} (i) Fix $\lil$ and consider any $\lambda'$ such that $|\lambda-\lambda'|\leq 1$. Then
\begin{align*}
1&=|f_\lambda(\lambda)-f_\lambda(\lambda')|\leq
\sup_{\zeta\in D(\lambda,1)}|f_\lambda^\prime(\zeta)| |\lambda-\lambda'|\ .
\end{align*}
By the Cauchy estimates
\begin{align*}
|f_\lambda^\prime(\zeta)|&\leq \sup_{\eta : |\eta-\zeta|=1} |f_\lambda(\eta)|\leq\sup_{z\in D(\lambda,2)} |f_\lambda(z)|
\leq\sup_{z\in D(\lambda,2)} A e^{-B|\Im \lambda|} e^{-C\alpha|\Im z|}\lesssim e^{-(B+C)|\Im \lambda|}\ ,
\end{align*}
which yields the desired estimate.

(ii) This is an immediate consequence of (i): if $C\geq 2\alpha$ we have
\begin{align*}
\sum_{\lil}\frac{ e^{-C|\Im \lambda|}}{1+|z-\lambda|^\beta}\lesssim
\sum_{\lil}\int_{D_\lambda}\frac{e^{-(C-2\alpha)|\Im \zeta|}}{1+|z-\zeta|^\beta}
dm(\zeta)\lesssim \int_{\C}\frac{dm(\zeta)}{1+|z-\zeta|^\beta}\ .
\end{align*}
\end{proof}

\begin{proof}[Proof of Theorem A]
$(a)\Rightarrow(b)$. Obvious.

$(b)\Rightarrow(c)$. Here we use the scheme of \cite[Lemma 2.2.6.]{Br-Ga}. Consider the  complete metric space $S=\bigl\{v=\{v_\lambda\}_\lambda : \|v\|_\infty\leq 1\bigr\}$,  and for $n\in\mathbb N$ let:
\[
S_n=\bigl\{\{v_\lambda\}_\lambda\in S : \exists f\in\B : f(\lambda)=v_\lambda, \ \lil, \textrm{and}\ 
\sup_{z\in\C} |f(z)|e^{-n|\Im z|}\leq n\bigr\}\ .
\]
By hypotheses $S=\cup_n S_n$ and by Baire's Cathegory Theorem there exists $n\in\mathbb N$ such that 
$\stackrel{\circ}{S_n}\neq\emptyset$. Since $0\in S_n$ for all $n$, there exists $\varepsilon\in(0,1)$ such that 
\[
\|v\|_\infty\leq\varepsilon\ \Longrightarrow\ v\in S_n\ .
\]
Let $v^\lambda=\{\varepsilon\delta_{\lambda,\lambda'}\}_{\lambda'}$. There exists then $g_\lambda\in\B$ with
$g_\lambda(\lambda')= \varepsilon\delta_{\lambda,\lambda'}$ and $\sup_z |g_\lambda(z)|e^{-n|\Im z|}\leq n$.
Taking $f_\lambda=g_\lambda/\varepsilon$ we get the desired properties.

$(c)\Rightarrow(d)$. Define $F(z)=\sum\limits_{\lil} w_\lambda(z) g_\lambda(z)$, where
\[
w_\lambda(z)=\left(\frac{\sin(z-\lambda)}{z-\lambda}\right)^3 e^{i\sigma_\lambda M(z-\lambda)}, \qquad
g_\lambda(z)=f_\lambda^2(z) \sin(z-\lambda)\ ,
\]
$M>0$ will be chosen later on 
and 
\[
\sigma_\lambda=
\begin{cases}
\ \ 1\quad &\textrm{if $\Im\lambda <0$}\\
-1 \quad &\textrm{if $\Im\lambda >0$.}
\end{cases}
\]
Notice that $w_\lambda$ is a holomorphic weight in $\B$ with $w_\lambda(\lambda)=1$, while $g_\lambda$ is a function in $\B$ vanishing at order 1 on $\lambda$ and at least at order 2 on all $\lambda'\neq \lambda$.

Let's see first that $F\in\B$. A direct estimate shows that 
\begin{align*}
|w_\lambda(z)|&\lesssim \frac{e^{(M+3)|\Im z|} e^{-(M-3)|\Im \lambda|}}{1+|z-\lambda|^3}\ ,\\
|g_\lambda(z)|&\lesssim e^{(2C+1)|\Im z|} e^{(2B+1)|\Im \lambda|}\ ,
\end{align*}
and therefore
\[
|F(z)|\lesssim e^{(M+2C+4)|\Im z|}\sum_{\lil} \frac{e^{-(M-2B-4)|\Im \lambda|}}{1+|z-\lambda|^3}\ .
\]
Taking $M$ big enough and applying Lemma~\ref{weak-separation}(ii) we see that $F\in\B$.

It is clear that $F(\lambda)=0$ for all $\lil$. On the other hand
\[
F'(\lambda)=\sum_{\lambda'\in\Lambda}w_{\lambda'}(\lambda) g_{\lambda'}^\prime(\lambda)=
w_{\lambda}(\lambda) g_{\lambda}^\prime(\lambda)=1\ .
\]

$(d)\Rightarrow(a)$. Let $\{v_\lambda\}_{\lil}$ satisfy \eqref{values} and define
\[
f(z)=\sum_{\lil} v_\lambda w_\lambda(z) \frac{F(z)}{F'(\lambda)(z-\lambda)}\ ,
\]
where $w_\lambda$ is defined as before. 

It is clear that $F(\lambda)=v_\lambda$ for all $\lil$, and a similar computation to the one in the previous implication shows that $f\in\B$, since $|F(z)/(F'(\lambda)(z-\lambda))|\lesssim e^{C(|\Im z|+|\Im \lambda|)}$.
\end{proof}

We finish this section by showing that a weakly separated union of two $\B$-interpolation sequences is still $\B$-interpolation.

\begin{theorem}\label{union}
 Let $\Lambda_1,\Lambda_2\in \Int\, \mathcal B$ be such that $\Lambda_1\cup\Lambda_2$ is weakly separated.  Then $\Lambda_1\cup\Lambda_2\in\Int\, \mathcal B$.
\end{theorem}

For the proof of this result we will need the following lemma.

\begin{lemma}
 Let $\Lambda\in\Int\, \mathcal B$ and let $z_0\notin \Lambda$. There exist $A,B,C>0$ independent of $z_0$ and $f\in \mathcal B$ such that
\begin{itemize}
 \item [(i)] $f(z_0)=1$ and $f(\lambda)=0$ for all $\lil$.
 \item [(ii)] $\sup_{z\in\C}|f(z)|e^{-C|\Im z|}\leq B e^{-C|\Im z_0|}/d(z_0,\Lambda)$.
\end{itemize}
\end{lemma}

\begin{proof}
 Assume first that $z_0=0$. Since $\Lambda$ is weakly separated there exists $\delta>0$ be such that $|\lambda-\lambda'|\geq 10\delta$ for all $\lambda,\lambda'\in\Lambda\cap\{|\Im z|<1\}$. 
 
\begin{lemma}\label{F}
There exists $F\in\B$ with $F(0)=0$ and such that 
\begin{align*}
 |F(\lambda)|&\geq \varepsilon\qquad\quad\textrm{if $|\lambda|\geq \delta$}\\
 |F(\lambda)|&\geq \varepsilon|\lambda|\qquad\textrm{if $|\lambda|< \delta$.}
\end{align*}
\end{lemma}

We deferr the proof of this technichal result to the end of the section.

Consider the values $v_\lambda=1/F(\lambda)$. Since $\sup_{\lil}|v_{\lambda}|\leq 1/\min(\varepsilon,\varepsilon d(0,\Lambda))$, there exist $A,B>0$ and $G\in\mathcal B$ such that $G(\lambda)=v_\lambda=1/F(\lambda)$ and $\sup_{z\in\C} |G(z)|e^{-A|\Im z|}\leq B/(\varepsilon d(0,\Lambda))$.

The function $f(z)=1-F(z)G(z)$ satisfies then the required properties.

For $z_0\notin\Lambda$ arbitrary apply the previous case to the sequence $\Lambda-z_0$. Notice that $\Lambda-z_0\in\Int \B$, since there exist the peak-functions required by Theorem A(c): just take $f_{\lambda-z_0}(z)=f_\lambda(z+z_0)$, where $f_\lambda$ are the peak-functions associated to $\Lambda$.

\end{proof}

\begin{proof}[Proof of Theorem~\ref{union}] Let us show that there exist peak-functions as in Theorem A(c). 

Fix $\lambda_1\in\Lambda_1$. Applying the previous Lemma to $\Lambda_2$ and $z_0=\lambda_1$ we see that there exist $A,B>0$ and $g_{\lambda_1}\in\B$ such that
\begin{align*}
 & g_{\lambda_1}(\lambda_1)=1\ ,\quad\textrm{and}\quad g_{\lambda_1}(\lambda_2)=0\ \forall\lambda_2\in\Lambda_2\\
&\sup_{z\in\C}|g_{\lambda_1}(z)|e^{-A|\Im z|}\leq B\frac{e^{C|\Im \lambda_1|}}{d(\lambda_1,\Lambda_2)}\leq B e^{C'|\Im \lambda_1|}\ .
\end{align*}
Since $\Lambda_1\in\Int \B$ there exist a functions $f_{\lambda_1}\in\B$, as given by Theorem A(c).
Then $h_{\lambda_1}=f_{\lambda_1}g_{\lambda_1}$ is a peak-function for $\Lambda_1\cup\Lambda_2$.

\end{proof}

\begin{proof}[Proof of Lemma~\ref{F}]
Consider
\[
\sin(\pi z)=\pi z\prod_{k\in\mathbb Z\setminus\{0\}}\left(1-\frac zk\right)\ .
\]
If $|\lambda-k|\geq \delta$ for all $\lil$ and $k\in\mathbb Z$, we can just take $F(z)=\sin(\pi z)$. 

For each $k\in\mathbb N$ such that there this $\lil$ with $|\lambda- k|< \delta$ or $|\lambda+ k|< \delta$ we replace in $\sin(\pi z)$ the factors $1-z/k$, $1+z/k$  by small perturbations $1-z/p_k$, $1-z/p_{-k}$, where
$p_k=k+i\epsilon_k$ and $\epsilon_k>0$ is chosen so that  
\begin{itemize}
\item[(a)] $|\epsilon_k|\leq 5\delta$

\item[(b)] $\epsilon_{-k}=\epsilon_k$

\item[(c)] $|p_{k}-\lambda|\geq\delta$ and $|p_{-k}-\lambda|\geq\delta$ for all $\lil$.
\end{itemize}
Denoting by $\mathcal I$ the indices of $\mathbb Z\setminus\{0\}$ affected by this modification, we define
\[
F(z)=\sin(\pi z)  \prod_{k\in\mathcal I} \frac{1-z/p_k}{1-z/k} \ .
\]
Notice that the zeros of this new sine-type function are at a distance of $\Lambda$ no smaller that $\delta$.

In order to see that $F\in\B$ is is enough to show that the latter product is bounded above and below away from the points $k$ and $p_k$. To do so we group the factors corresponding to $k$ and $-k$. Since $p_{-k}=-\overline{p_k}$ we have
\begin{align*}
\prod_{k\in\mathcal I}\frac{1-z/p_k}{1-z/k}&=
\prod_{k\in\mathcal I} \frac{z-p_k}{z-k} \frac k{p_k}=
\prod_{k\in\mathcal I\cap\mathbb N} \frac{(z-p_k)(z-p_{-k})}{(z-k)(z+k)}\frac{-k^2}{p_k p_{-k}}\\
&=\prod_{k\in\mathcal I\cap\mathbb N} \left( 1-\frac{\epsilon_k(\epsilon_k+2iz)}{z^2-k^2}\right)
\frac 1{1+\epsilon_k^2/k^2}\ ,
\end{align*}
from which the required properties follow easily. 
\end{proof}

\section{Proof of the geometric conditions}\label{sufsection}

In this section we briefly recall the proof of Theorem B.

\subsection{Necessity}
The necessity of (i) is a standard (and straightforward) application of Jensen's formula to the functions $f_\lambda$ provided by Theorem A(c) on the disks $D(\lambda,|\Im\lambda|)$.

To see that (ii) is also necessary consider the functions $g_\lambda(z)=f_\lambda(z) e^{iCz}$, which according  to Theorem A(c) are bounded
in the upper half-plane $\H$. Applying Jensen's formula  to $g_\lambda$ in $\H$ we have
\[
 \sum_{\lambda'\neq\lambda: \Im \lambda'>0}\log\left|\frac{\lambda-\lambda'}{\lambda-\bar\lambda'}\right|^{-1}\lesssim\Im \lambda
\qquad \textrm{for all} \quad \lambda \in\Lambda\cap\H\ .
\]
The estimate $1-t\leq\log t^{-1}$ for $t\in(0,1)$ yields then
\[
\sup_{\lambda\in\Lambda\cap\H} \sum_{\lambda'\neq\lambda: \Im \lambda'>0}\frac{\Im\lambda'}{|\lambda-\bar\lambda'|^2}< +\infty\ .
\]

Given $x\in\R$, consider the $\lambda\in\Lambda\cap\H$ closest to $x$. Then $\frac{\Im\lambda'}{|x-\bar\lambda'|^2}\leq 2
\frac{\Im\lambda'}{|\lambda-\bar\lambda'|^2}$. This and the corresponding computation for the lower half-plane imply (ii).

\subsection{Sufficency}

According to Theorem~\ref{union} it will be enough to show that $\Lambda_+=\Lambda\cap\{\Im z>0\}$ and $\Lambda_-=\Lambda\cap\{\Im z<0\}$ are $\B$-interpolation.

By Theorem~A(d), in order to prove that $\Lambda_+$ is $\B$-interpolation it
is enough to construct a function  $F\in \B$  such that $\Lambda_+\subset \mathcal
Z(F)$ and
\[
|F'(\lambda)|\geq \varepsilon e^{-K \Im \lambda}
\qquad \lambda\in\Lambda_+ 
\]
for some constants $\varepsilon, k>0$. 

Start with any entire function $G$ with zero ser $\mathcal Z(G)=\Lambda_+ $. Condition (ii)
implies that $\Lambda_+$ is a Blaschke sequence in $\H$,  i.e. the Blaschke product 
\[
B(z)=\prod_{\lambda\in\Lambda_+}\frac{z-\lambda}
{z-\bar\lambda}, \qquad z\in\H 
\]
converges.
Define
\[
\Psi(z)=N|\Im\, z|-\phi(z)\ ,
\]
where
\begin{equation*}
\phi(z)=
\begin{cases}
\displaystyle\log\Bigl|\frac{G(z)}{B(z)}\Bigr|\quad &\textrm{ $\Im\, z>0$}\\
\log|G(z)|\quad &\textrm{ $\Im\, z\leq0$.}
\end{cases}
\end{equation*}

A computation shows that
$\phi$ is harmonic outside the real axis, subharmonic on $\C$ and $\Delta\phi(x)=\sum\limits_{\lambda\in\Lambda_+}\frac{\Im\lambda}{|x-\lambda|^2}\, dx$ (see \cite[Lemma 10]{MOO}).
Thus, by condition (ii),
$\Delta\Psi\simeq dx$ when  $N\in\N$ is big enough. In this
situation, according to \cite[Lemma~3]{Or-Se},  there exists a multiplier
associated to $\Psi$, i.e., an entire function $h$ such that:
\begin{itemize}
\item[(a)] $\mathcal Z (h)$ is a separated sequence contained in $\R$ and separated from $\Lambda$ (i.e. 
$\inf\limits_{z\in\mathcal Z (h)} d(z,\Lambda)>0$).
\item[(b)] Given any $\varepsilon>0$, 
$|h(z)|\simeq \exp(\Psi(z))$ for all points $z$ such that $d(z,\mathcal Z(h))
>\varepsilon$.
\end{itemize}

Define now $F= h G$. It is clear that $F\in \B$:
\[
|F(z)|\lesssim e^{\Psi(z)+\log |G(z)|}\leq e^{\Psi(z)+\phi(z)}\leq 
e^{N p(z)}\qquad z\in\C .
\]
It is also clear that $\Lambda_+\subset\mathcal Z(F)$, since
$\Lambda_+\subset\mathcal Z(G)$. 

In order to prove that there exist $\varepsilon,C>0$ such that
\begin{equation}\label{cota} 
|F'(\lambda)|\geq\varepsilon
e^{-C \Im \lambda}
\end{equation}
consider the disjoint disks $D_\lambda=D(\lambda,\delta_\lambda)$,
$\delta_\lambda=\delta e^{-C\frac{p(\lambda)}{m_\lambda}}$
provided by Lemma~\ref{weak-separation}(i).
Since $\Lambda_+$ is far from $\mathcal Z(h)$, the estimate
\[
|G(z)|=|h(z)|e^{\phi(z)}|B(z)|\simeq e^{N|\Im\, z|}  |B(z)| 
\qquad z\in\partial D_\lambda
\]
holds.

From here we finish as in \cite{MOO}: since the hypotheses imply that
there exists $C>0$ such that $|B(z)|\geq\epsilon e^{-C |\Im z|}$, 
$z\in\partial D_\lambda$, we have
$|F(z)| \gtrsim  e^{-C p(z)} $ for all 
$z\in\partial D_\lambda$.
The function $g(z)=F(z)/(z-\lambda)$ is then 
holomorphic, non-vanishing in $D_\lambda$, and with $|g(z)|\gtrsim e^{-c
\Im \lambda}$ for $z\in\partial D_\lambda$. By the minimum principle 
\[
|F'(\lambda)|=|g(0)|\gtrsim e^{-c\Im \lambda},
\] 
as desired.


\providecommand{\bysame}{\leavevmode\hbox to3em{\hrulefill}\thinspace}
\providecommand{\MR}{\relax\ifhmode\unskip\space\fi MR }

\end{document}